\documentclass
{amsart}

\usepackage[latin1]{inputenc}

\usepackage{amsfonts}
\usepackage{latexsym}
\usepackage{amssymb}
\theoremstyle{plain}
\newtheorem{theorem}[subsection]{Theorem }
\newtheorem{definition}[subsection]{Definition }
\newtheorem{Pro}[subsection]{Proposition }
\newtheorem{remark}[subsection]{Remark}

\newtheorem{fact}[subsection]{Fact}

 \numberwithin{equation}{section}

\def\N{{\mathbb N}}

\def\S{{\mathbb S}}
\def\cont{\mathfrak c}
\begin{document}
\author[ D. Dikranjan,  E. Mart\' \i n-Peinador and V. Tarieladze]{\sc D. Dikranjan,  E. Mart\' \i n-Peinador and V. Tarieladze}

\title[Countable powers of compact abelian groups]{Countable powers of compact abelian groups in the uniform topology and cardinality of their dual groups}



\begin{abstract}

For a topological Abelian group $X$ we consider in the group $X^{\N}$ the uniform topology and
study some properties of the obtained topological group. We show, in particular, that if $X=\mathbb S$ is the  circle group, then the group
$\mathbb S^{\N}$ endowed with the uniform topology has  dual group with cardinality $2^{\cont}$.

\end{abstract}
\maketitle
\

\section{Introduction}
For a set $A$ we denote by ${\rm{Card}}(A)$ or by $|A|$ the cardinality of $A$.
\par
For a topological space $A$ we denote by:
\begin{itemize}
\item
${ w}(A)$  the {\em weight} of $A$, i.e. the smallest cardinality of a base of A,
\item
${d}(A)$  the {\em density character}  of $A$, i.e., the smallest  cardinality of a dense subset of $A$,
\item
${ c}(A)$ the {\em cellularity} of $A$, i.e., the smallest cardinal $\kappa$ such that every family of
non-empty pairwise disjoint open sets has cardinality $\leq \kappa$.
\end{itemize}

All groups considered in this note will be Abelian.

Let   $G,Y$ be groups.  We denote by ${\rm{Hom}}(G,Y)$ the group of all group homomorphisms from
$G$ to $Y$. If $G,Y$ are topological groups,  ${\rm{CHom}}(G,Y) $ stands for the continuous elements  of ${\rm{Hom}}(G,Y)$.

\par
A set $\Gamma\subset {\rm{Hom}}(G,Y)$ will be called {\it separating}, if
$$
(x_1,x_2)\in G\times G,\,x_1\ne x_2\,\Longrightarrow \exists \gamma\in \Gamma,\,\gamma(x_1)\ne\gamma (x_2)\,.
$$
\par
In what follows the letter  $\mathbb S$ will stand for the multiplicative group of complex numbers of modulus one  endowed with the usual compact topology.
We write:

$$\S_+=\{s\in \S:\rm{Re}(s)\ge 0\}\,.$$
For a topological group $G$ we denote by $\mathcal N(G)$ the collection of all neighborhoods of the neutral element of $G$.
\par
A subset $K$ of a topological  group $G$ is called {\it precompact} if for every $V\in \mathcal N(G)$  there exists a finite non-empty $A\subset G$ such that $K\subset A+V(:=\{a+v:a\in A,v\in V\})$.
\par
A topological  group $G$ is called {\it  locally precompact} if  $\mathcal N(G)$ admits a basis consisting of precompact  subsets of $G$.
\par
For a group $G$ an element of ${\rm{Hom}}(G,\mathbb S)$ is called a ({\it  multiplicative}) {\it character}.
\par
For a topological  group $G$ we write
$$
G^\wedge:={\rm{CHom}}(G,\mathbb S)\,.
$$
The elements of $G^\wedge$ are called {\it  continuous characters} and $G^\wedge $ itself is {\it the (topological)  dual} of $G$.\\
Accordingly,
for a group $G$ and for a group topology $\tau$ in $G$ we denote by $(G,\tau)^\wedge$ the dual of the topological group $(G,\tau)$, i.e.,
$$
(G,\tau)^\wedge=\{{\chi\in \rm{Hom}}(G,\mathbb S): \chi\,\,\text{is}\,\,\tau-\text{continuous}\}\,.
$$
\par
A topological group $G$ is called {\it maximally almost periodic},
for short a MAP-group, if $G^\wedge $ is separating. It is known that every locally precompact Hausdorff topological  group is a MAP-group (this follows from the highly non-trivial  statement that every locally compact Hausdorff topological  group is a MAP-group \cite[Theorem 22.17]{HR}). It is relatively easy to show  that every discrete group (i.e. a group endowed with the discrete topology) is MAP; in this case the cardinality of the dual group can be calculated as follows. \begin{fact}\label{kak}{\em (\cite{K43}; see also \cite[(24.47)]{HR})}
Let $D$ be an infinite discrete group. Then $|D^\wedge|=2^{|D|}$.
\end{fact}

Let us recall the definition of a locally quasi-convex group.

\begin{definition}\label{qcx} {\em \cite{Vil}}
 A subset $A$ of a {\bf topological} group $G$ is called {\bf quasi-convex} if for every $x\in G\setminus A$ there exists
$\chi\in G^\wedge$ such that
$$
\chi(A)\subset \S_+,\,\,\,\text{but}\,\,\chi(x)\not\in \S_+\,.
$$
 A { topological} group $G$ is called {\bf locally quasi-convex} if $\mathcal N(G)$ admits a basis consisting of quasi-convex subsets of $G$.
\end{definition}

Similar concepts (without a reference to \cite{Vil}) were defined later in \cite{Su}, where  the terms {\it   polar set} and {\it   locally polar group} are used instead of   'quasi-convex set'  and   'locally  quasi-convex group'. It is easy to see that every locally quasi-convex  Hausdorff topological abelian group is a MAP-group.
  For more information about  locally quasi-convex
 topological groups we refer \cite{A, B, CMPT}.
\par
The locally  precompact groups are a prominent class of locally quasi-convex groups.
\par
If $G$ is a   compact Hausdorff topological abelian group, then ${ w}(G)=\rm{Card}(G^{\wedge})$ \cite[Theorem 24.15]{HR}. In particular, an infinite    compact Hausdorff topological abelian group $G$ is metrizable iff $\rm{Card}(G^{\wedge})=\aleph_0$.
The following  statement characterizes the locally  precompact groups with countable dual.

\begin{Pro}\label{codu1}{\em \cite{DMT}}
For an infinite locally precompact Hausdorff topological abelian group $G$ TFAE:
\begin{itemize}
\item[(i)]  $G$ is precompact metrizable.
\item[(ii)] $G^{\wedge}$ is countably infinite.
\end{itemize}
\end{Pro}

 The implication $(ii)\Longrightarrow (i)$ of Proposition \ref{codu1} may fail if $G$ is
 a locally quasi-convex Hausdorff group, which is not locally precompact (see Proposition \ref{dikrfor2}).

 In Section \ref{Sec1} we introduce the group of sequences ${X}^\mathbb N$ for a  topological group $X$ and
 its uniform topology $\mathfrak{u}$.  The following statement is the main result of this note.

 \begin{theorem}\label{me-th}
 Let $X\ne \{0\}$ be a compact Hausdorff topological abelian group and $G=({X}^\mathbb N,\mathfrak{u})$.
Then
\begin{itemize}
\item[$(a)$] $|{\rm{CHom}} (G,X)|\ge 2^{\cont}$.
\item[$(b)$]  If $|X|\leq 2^{\cont}$ (in particular, if $X$ is separable), then  $|{\rm{CHom}} (G,X)|= 2^{\cont}$.
\item[$(c)$] If $X={\mathbb S}$, then $|G^{\wedge}|= 2^{\cont}$.
\end{itemize}

 \end{theorem}



\bigskip

The proof of this theorem is given in Section 3.
\begin{remark}{\em In connection with Theorem \ref{me-th} we note that in \cite{FK} it is  presented the first example of a Banach space $G$ over $\mathbb R$ with the following properties: $|G|=\cont$ and $|{\rm{CHom}} (G,\mathbb R)|= 2^{\cont}$.
}
\end{remark}

\section{Two  groups of  sequences and their uniform topology}\label{Sec1}

Let $X$ be a set. As usual, $X^\mathbb N$ will denote the set of all sequences ${\bf x}=(x_n)_{n\in \mathbb N}$ of elements of
$X$.
\par
If $X$ is a group with the neutral element $\theta$, then
\par
$X^{(\mathbb N)}$ will stand for the subgroup of $X^\mathbb N$
consisting of all sequences from $X^\mathbb N$, which eventually
equal to $\theta$.
 If $X$ is a topological group with the neutral element $\theta$, then
$$
 c_0(X):=\{ (x_n)_{n\in \mathbb N}\in  X^\mathbb N :\lim_nx_n=\theta\, \}.
 $$
\par
 Clearly $c_0(X)$ is a subgroup of $X^\mathbb N$, $X^{(\mathbb N)}\subseteq c_0(X)$ and $X^{(\mathbb N)}= c_0(X)$ iff $X$ has only trivial convergent sequences.

In what follows $X$ will be a fixed Hausdorff topological  group.
\par
We denote by $\mathfrak{p}_X$ the product topology in $X^{\N}$ and by $\mathfrak{b}_X$ the box topology in $X^{\N}$. It is easily verified that the collection
$$
\{ V^{\N}:V\in \mathcal N(X)\}
$$
 is a basis at $e:=(\theta,\theta,\dots)$ for a group topology in $X^{\N}$ which we denote by  $\mathfrak{u}_X$. 
 In all three cases we shall omit the subscript $_X$ when no confusion is likely.

   The topology  $\mathfrak{u}$ in $X^{\N}$ is nothing else but the topology of uniform convergence on $\N$ when the elements of $X^{\N}$ are viewed
   as functions from $\N$ to $X$ and $X$ is considered as a uniform space with respect to its left (=right) uniformity \cite{Bour}. So it will be called  {\it the uniform  topology}.  Since it  plays an important role in the sequel, we give in the next proposition an account of its main properties.

 We write:
$$
\mathfrak{p}_0:=\mathfrak{p}|_{c_0(X)}\;\; , \;\;\mathfrak{b}_0:=\mathfrak{b}|_{c_0(X)}\;\; \mbox{ and }\;\; \mathfrak{u}_0:=\mathfrak{u}|_{c_0(X)}.
$$

 \begin{Pro}\label{descr1} {\em \cite{DMT}} Let $(X,+)$ be a Hausdorff topological abelian group.
 \begin{itemize}
   \item[(a)]  The uniform topology $\mathfrak{u}$    is a Hausdorff group topology in $X^\mathbb N$ with $\mathfrak{p}\le \mathfrak{u}\le \mathfrak{b}$. Moreover,
     \begin{itemize}
         \item[(a$_1$)]  $\mathfrak{p}|_{X^{(\mathbb N)}}=\mathfrak{u}|_{X^{(\mathbb N)}}\,\Longleftrightarrow\, X=\{0\}$.
         \item[(a$_2$)] $\mathfrak{u}|_{X^{(\mathbb N)}}=\mathfrak{b}|_{X^{(\mathbb N)}} \Longrightarrow\,
          X\,\text{is\, a \rm{P}-group}\Longrightarrow\,\mathfrak{u}=\mathfrak{b}$; in particular, if $X$ is metrizable and $\mathfrak{u}|_{X^{(\mathbb N)}}=\mathfrak{b}|_{X^{(\mathbb N)}}$, then $X$ is discrete.
      \end{itemize}
   \item[(b)]  The passage from $X$ to $(X^\mathbb N,\mathfrak{u})$ preserves (sequential) completeness,  metrizability, MAP and local quasi-convexity.
   \item[(c)]   If $X\ne\{0\}$ and $G:=(X^\N,\mathfrak{u})$, then:
      \begin{itemize}
        \item[(c$_1$)]  $c(G)\ge \cont$, in particular $G$ is not separable;
 \item[(c$_2$)]  $(X^{(\N)}, \mathfrak{u}|_{X^{(\N)}})$ is not precompact and hence, $(c_0(X),\mathfrak{u}_0)$ and  $(X^\mathbb N,\mathfrak{u})$ are not precompact.
       \end{itemize}
\end{itemize}
   \end{Pro}

 \begin{remark}{\em The topology of $({\mathbb S}^\mathbb N,\mathfrak{u})$ can be induced by the invariant metric $\rho$ defined by the equality
$$
\rho({\bf x},{\bf y})=\sup_{n\in \mathbb N}|x_n-y_n|,\quad {\bf x},{\bf y}\in {\mathbb S}^\mathbb N\,.
$$
 The metric group $G:=({\mathbb S}^\mathbb N,\rho)$ was considered earlier in \cite[Example 4.2]{CoRo66}, where it was noted that $G$
  is not precompact, but every uniformly  continuous  $f:G\to \mathbb R$ is bounded. }
\end{remark}

\begin{Pro}\label{descr2m}{\em \cite{DMT}} Let $X$ be a Hausdorff topological abelian group.
 \begin{itemize}
   \item[(a)]  $(c_0(X), \mathfrak{u}_0)$  is a Hausdorff topological group  having as  a basis at zero the collection $\{ V^{\N}\cap c_0(X) :V\in \mathcal N(X)\}$.
   \item[(b)]  $\mathfrak{p}_0 \le \mathfrak{u}_0\le \mathfrak{b}_0$. Moreover, $\mathfrak{p}_0 =\mathfrak{u}_0\,\Longleftrightarrow\, X=\{0\}$; if $X$ is metrizable and $\mathfrak{u}_0 =\mathfrak{b}_0$, then $X$ is discrete.
   \item[(c)]   The passage from $X$ to $(c_0(X),\mathfrak{u}_0)$ preserves (sequential) completeness, metrizability, separability, MAP, local quasi-convexity,  non-discreteness, and connectedness.
  \end{itemize}
  \end{Pro}

   \begin{remark}{\em Let $X$ be the additive group  $\mathbb R$ with the usual topology.
   \par
   (1) By Proposition \ref{descr1} $(\mathbb R^{\mathbb N},\mathfrak{u})$ is a complete metrizable topological abelian group. Note that although $\mathbb R^{\mathbb N}$ is a vector space over $\mathbb R$, $(\mathbb R^{\mathbb N},\mathfrak{u})$ {\bf is not} a topological vector space
 over $\mathbb R$. The group  $(\mathbb R^{\mathbb N},\mathfrak{u})$ is not connected; the connected component of the null element coincides with $l_{\infty}$ and the topology $\mathfrak{u}|_{l_{\infty}}$ is the usual Banach-space topology of $l_{\infty}$.
 \par
 (2) By Proposition \ref{descr2m}  $(c_0(\mathbb R),\mathfrak{u}_0)$ is a complete separable metrizable topological abelian group. Note that  $c_0(\mathbb R)$ is a vector space over $\mathbb R$ and  $(c_0(\mathbb R),\mathfrak{u}_0)$ {\bf is}  a topological vector space
 over $\mathbb R$. The topology $\mathfrak{u}_0$ is the usual Banach-space topology of $c_0$.
 \par
 (3) It is easy to see that $\mathbb Z^{(\mathbb N)}$ is a closed subgroup of  $(c_0(\mathbb R),\mathfrak{u}_0)$ and the quotient group
 $$(c_0(\mathbb R),u_0)/\mathbb Z^{(\mathbb N)}$$
  is topologically isomorphic with $(c_0(\mathbb S),\mathfrak{u}_0)$.
   }
   \end{remark}

   \begin{remark}\label{Ro}{\em The topology of $(c_0(\mathbb S),\mathfrak{u}_0)$ can be
induced by the invariant metric $\rho_0$ defined by the equality
$$
\rho_0({\bf x},{\bf y})=\sup_{n\in \mathbb N}|x_n-y_n|,\quad {\bf x},{\bf
y}\in c_0(\mathbb S)\,.
$$
  It seems that the metric   group $(c_0(\mathbb S),\rho_0)$ was first considered by   Rolewicz in \cite{Rol}, where he proves that it is a monothetic group.  As he underlines,  a monothetic and completely metrizable group need not be compact  or discrete, a fact   which  breaks the dichotomy existing in the class of LCA-groups: namely,  a monothetic LCA-group must be  either compact or discrete (\cite[ Lemme 26.2 (p. 96)]{W}; see also \cite[Remark 5]{anz}, where a construction of a different example of a complete metrizable monothetic non-locally compact group is indicated).
 In \cite{Ni} it is observed that $|(c_0(\mathbb S), d_0)^{\wedge}|=\aleph_0$.\\
 A proof of the fact that $(c_0(\mathbb S),\rho_0)$  is monothetic  and $|(c_0(\mathbb S), \rho_0)^{\wedge}|=\aleph_0$ is contained also in \cite[pp. 20--21]{DPS}. In  \cite{Gabr} it is shown further that $(c_0(\mathbb S),\rho_0)$ is a Pontryagin reflexive group. }
\end{remark}

 The following statement provides, in particular, a wide class of non-compact Polish locally quasi-convex topological abelian groups with countable dual .
   \begin{Pro}\label{dikrfor2}{\em \cite{DMT}} For an infinite  locally compact Hausdorff topological abelian group $X$ TFAE:
\begin{itemize}
     \item[(i)]  $X$ is  compact connected and metrizable.
     \item[(ii)]  $|(c_0(X), \mathfrak{u}_0)^{\wedge}|=\aleph_0$.
\end{itemize}
\end{Pro}

Our  Theorem \ref{me-th} shows, in particular, that in the implication $(i)\Longrightarrow (ii)$ of Proposition \ref{dikrfor2} the group
 $(c_0(X), \mathfrak{u}_0)$ cannot be replaced by the group $(X^\mathbb N,\mathfrak{u})$ .

\section{Auxiliary statements and proof of Theorem \ref{me-th}}

We will need  the following refinement of  item (c$_1$) of Proposition \ref{descr1}.

  \begin{Pro}\label{NEW}
 Let $X$ be a Hausdorff topological  group and $G:=(X^\N,\mathfrak{u})$. Then $d(G)\leq d(X)^{\aleph_0}$.   \end{Pro}

 \begin{proof}
Let $D$ be  a dense subset of $X$ of size $d(X)$. It suffices to show that $D^\N$ is dense in $G$. This follows immediately
from the definitions.   \end{proof}

We will also use later   the following known statement.
\begin{Pro}\label{NEW1}
Let $X$ be a compact Hausdorff topological  group.
\par
$(a)$ If  $|X|\le 2^\cont$, then $d(X)\le \cont$.
\par
$(b)$ If $X$ is infinite, then $|X| = 2^{w(X)}$.
 \end{Pro}
\begin{proof}
$(a)$.
From $w(X)\le |X|$ (see \cite[(3.1.21)]{Eng}) and  $|X|\le 2^\cont$  we have: $w(X)\le 2^\cont$.
From the last inequality according to Ivanovskii-Kuzminov theorem \cite[Theorem 4.1.7 (p. 222)]{AT} we get the existence of a continuous surjection $f:\{0,1\}^{ 2^\cont}\to X$.
By \cite[(2.3.15)]{Eng} $d(\{0,1\}^{ 2^\cont})\le \cont$. Consequently, $d(X)=d(f(\{0,1\}^{ 2^\cont}))\le d(\{0,1\}^{ 2^\cont})\le \cont$.

$(b)$
Let $D$ be the group $X^{\wedge}$ endowed with the compact-open topology. Then it is not hard to see that $D$ is a discrete group.
  Endow $D^{\wedge}$ with the compact-open topology (which, in this case, coincides with the topology of point-wise convergence). It follows easily from Tikhonov-product theorem that $D^{\wedge}$ is compact Hausdorff topological group.
The powerful Pontryagin duality theorem implies that the compact groups $D^{\wedge}$ and $X$ are topologically isomorphic:  $D^{\wedge}\cong X$.
By \cite[(24.15)]{HR} we have: $w(X)=|D|$. Since $X$ is infinite, $D$ is infinite as well; from this and $X\cong D^{\wedge}$ by Fact \ref{kak} we have: $|X| = 2^{|D|}$. From this equality, as $w(X)=|D|$, we get $(b)$.
\end{proof}

%
 In the sequel we deal with cardinals larger than $\cont$ and even $2^\cont$.
We need to recall several standard definitions on cardinals.

\begin{definition} Let $\kappa$ be a cardinal.
\begin{itemize}
\item
 $\kappa^+$ will denote the successor  of $\kappa$.
\item
 $\kappa$ is called a {\em strong limit cardinal}, if $2^\lambda < \kappa$ for all $\lambda < \kappa$.
\end{itemize}
\end{definition}

Clearly, a strong limit cardinal $\kappa$ is a limit cardinal (i.e., $\kappa$ is not of the form  $\lambda^+$ for any cardinal $\lambda$). Obviously, $\aleph_0$ is a strong limit cardinal.
To obtain the next strong limit cardinal one has to go a long way. To this end let
$$
\beth_0 = \aleph_0\ \mbox{ and }\beth_{n+1}=2^{\beth_{n}} \ \mbox{  for all } \  n \in \N.
$$
Then $\beth_\omega := \sup_{n \in \N} \beth_{n}$ is the smallest uncountable strong limit cardinal. From Proposition \ref{NEW1}$(b)$ one can easily deduce that the cardinality of an infinite compact group is never a strong limit cardinal.
In particular, if $X$ is a compact Hausdorff topological  group, then  $|X|\ne \beth_\omega$.

\begin{Pro}\label{NEW2}
 Let $X$ be a Hausdorff topological  group and $G:=(X^\N,\mathfrak{u})$. Then:
 \par
 $(a)$ $
|{\rm{CHom}}(G,X)|  \leq  |X|^{d(X)^{\aleph_0}}$.
\par
$(b)$ If $X$ is compact and $|X|\le 2^\cont$, then $|{\rm{CHom}}(G,X)| \leq 2^\cont$.
\par
$(c)$ If $X$ is compact and $ 2^\cont<|X|<\beth_\omega$, then under the assumption of the Generalized Continuum  Hypothesis (GCH) we have:  $|{\rm{CHom}}(G,X)| \leq |X|$.
 \end{Pro}
  \begin{proof}
  $(a)$ Denote by ${\rm{C}} (G,X)$ the set of all continuous mappings $f:G\to X$;  it is easy to see that
  $
  |{\rm{C}} (G,X)|\le |X|^{d(G)}\,.
  $
  From this inequality and Proposition \ref{NEW} we get: $|{\rm{C}} (G,X)|\le |X|^{d(X)^{\aleph_0}} $.
  Consequently,
  $$
  |{\rm{CHom}} (G,X)|\le |{\rm{C}} (G,X)|\le |X|^{d(X)^{\aleph_0}}\,.
  $$
  \par
  $(b)$ From $|X|\le 2^\cont$ by Proposition \ref{NEW1}$(a)$ we have: $d(X)\le \cont$. From the inequalities $|X|\le 2^\cont$, $d(X)\le \cont$ and from $(a)$ we obtain: $|{\rm{CHom}}((G,X)|  \leq  |X|^{d(X)^{\aleph_0}} \le (2^\cont)^{\cont^{\aleph_0}}=2^\cont\,.$
  \par
  $(c)$ Let us recall that according to GCH, one has $2^\kappa = \kappa^+$ for every infinite cardinal $\kappa$.   In particular, every uncountable $\kappa< \beth_\omega$ has the form $\kappa= \beth_m$ for some $m\geq 1$. In particular, the hypothesis $ 2^\cont<|X|< \beth_\omega$ and the fact that $ 2^\cont = \beth_2$ imply that   $|X| = \beth_m$ for some $m>2$. By   Proposition \ref{NEW1}$(b)$, it follows  that $w(X)= \beth_{m-1}$, with $m-1 > 1$. So $w(X) = 2^{\beth_{m-2}}$ and consequently,
 $$
 w(X)^{\aleph_0}=(2^{\beth_{m-2}}) ^{\aleph_0} =2^{\beth_{m-2}}= \beth_{m-1}\ \mbox{ and }  \ d(X)^{\aleph_0} \leq w(X)^{\aleph_0} = \beth_{m-1}.
 $$
   Using the latter inequality and item (a), we get
$$
|{\rm{CHom}}(G,X)| \leq |X|^{ d(X)^{\aleph_0}}\leq  \beth_{m}^{ \beth_{m-2}} = (2^{ \beth_{m-1}})^{\beth_{m-2}}
=2^{ \beth_{m-1}\cdot  \beth_{m-2}}  = 2^{ \beth_{m-1}} = \beth_{m}= |X|.
$$
  \end{proof}

We prove in Theorem \ref{me-th} that in Proposition \ref {NEW2}$(b)$ is practically an equality.\\ The groups $X$ for which in Proposition \ref {NEW2}$(c)$ can occur an equality will be treated elsewhere.\\
 
 {\bf Proof of Theorem \ref{me-th}. }\\
  $(a)$.
Denote by  $\mathfrak F$ the set of all ultrafilters on $\mathbb N$.\\
It is known that
\begin{equation}\label{2sep0}
\rm{Card}(\mathfrak F)=2^{\cont}\, .
\end{equation}
 For a  filter $\mathcal F$ on $\mathbb N$, $(x_n)_{n\in \mathbb N}\in {X}^\mathbb N$ and $x\in X$ we write:
$$
\lim_{n,{\mathcal F}}x_n=x
$$
if for every $W\in \mathcal N(X)$ there is $F\in \mathcal F$ such that $x_n-x\in W,\,\forall n\in F$.

Since $X$ is compact Hausdorff, it follows that for every $\mathcal F\in \mathfrak F$ and $(x_n)_{n\in \mathbb N}\in {X}^\mathbb N$ there exists a unique $x\in X$ such that $
\lim_{n,{\mathcal F}}x_n=x$.
This follows from the fact that the sets $A_F:=\{x_n: n\in F\}$, when $F$ runs over $\mathcal F$,
give rise to a filter base on $X$ that generates an ultrafilter $\mathcal F^*$ on $X$ that has a unique limit point $x$.
\par
 For a filter  $\mathcal F\in \mathfrak F$ define the mapping $\chi_{{\mathcal F}}:{ X}^\mathbb N\to X$ by
equality:
$$
\chi_{\mathcal F}({\bf x})=\lim_{n,{\mathcal F}}x_n,\quad \forall {\bf x}=(x_n)_{n\in \mathbb N}\in {X}^\mathbb N\,.
$$
 Then
 \begin{equation}\label{2sep}
 \chi_{\mathcal F}\in {\rm{CHom}} (G,X)\quad \forall \mathcal F\in \mathfrak
 F\,.
 \end{equation}
When $\mathcal F$ is the principal ultrafilter on $\N$ generated by a fixed $n \in \N$, then  $ \chi_{\mathcal F}$ is simply
the projection $p_n: X^\N  \to X$ on the $n$-th coordinate, so (\ref{2sep}) is obvious.
 To verify (\ref{2sep}) in the general case, fix $\mathcal F\in \mathfrak
 F\,.$ As
$$
\chi_{\mathcal F}({\bf x}+{\bf y})=\lim_{n,{\mathcal
F}}(x_n+y_n)=
\lim_{n,{\mathcal F}}x_n+\lim_{n,{\mathcal
F}}y_n=\chi_{\mathcal F}({\bf x})+\chi_{\mathcal F}({\bf y}),\quad
\forall {\bf x},{\bf y}\in {X}^\mathbb N\,,
$$
we conclude that $\chi_{\mathcal F}\in {\rm{Hom}} (X^{\N},X)$. To see that $\chi$ is continuous on  $(X^{\mathbb N},\mathfrak{u})$, fix a
closed $W\in \mathcal N(X)$. Since $W$ is closed, for ${\bf x}=(x_n)_{n\in \mathbb N}\in {W}^\mathbb N$ we shall have
$$
\chi_{\mathcal F}({\bf x})=\lim_{n,{\mathcal F}}x_n\in W\,.
$$
Consequently, $\chi_{\mathcal F}(W^{\mathbb N})\subset W$. From this relation, as $W^{\mathbb N}\in \mathcal N(X^{\mathbb N},\mathfrak{u})$,
we get that $\chi_{\mathcal F}$ is continuous on  $(X^{\mathbb N},\mathfrak{u})$ and (\ref{2sep}) is proved.
\par
We  also have:
\begin{equation}\label{2sep2}
 \mathcal F_1\in \mathfrak F,\,\mathcal F_2\in \mathfrak F,\,\,\mathcal F_1\ne \mathcal F_2\, \Longrightarrow\,\chi_{\mathcal F_1} \ne \chi_{\mathcal F_2}
\end{equation}
 In fact, as $\mathcal F_1$ and $\mathcal F_2$ are {\it distinct} ultrafilters, there is $F\in  \mathcal F_1$ such that $F\not\in  \mathcal F_2$. Let  ${\bf x}=(x_n)_{n\in \mathbb N}\in
{X}^\mathbb N\,$ be defined by conditions: $x_n=0$ if $n\in F$ and $x_n=a\ne 0$ if $n\in \mathbb N\setminus F$. Then $\chi_{\mathcal
F_1}({\bf x})=0$ and $\chi_{\mathcal F_2}({\bf x})=a$. Therefore, $\chi_{\mathcal F_1} \ne \chi_{\mathcal F_2}$ and (\ref{2sep2}) is proved.
 \par
Clearly (\ref{2sep0}),(\ref{2sep})  and (\ref{2sep2}) imply that $|{\rm{CHom}} (G,X)|\ge 2^{\cont}$.
\par
$(b)$ follows from $(a)$ and from Proposition \ref{NEW2}$(b)$.\\
$(c)$ follows immediately from $(b)$.   \hfill $\Box$
\begin{remark}\label{OQ}{\em In notation of  Theorem \ref{me-th} and its proof  let
$$
\Gamma:=\{\chi_{\mathcal F}:{\mathcal F}\in \mathfrak F\}\,.
$$
Denote by  $\langle\Gamma\rangle$ the subgroup of ${\rm{CHom}} (G,X)$ generated by the set $\Gamma$.
As $|\Gamma|=2^\cont$, we have also that $|\langle\Gamma\rangle|=2^\cont$; so, in view of Theorem \ref{me-th}$(b)$, for $X$ with
$|X|\leq 2^{\cont}$ we have the equality:  $|\langle\Gamma\rangle|=2^\cont=|{\rm{CHom}} (G,X)|$.\\
We do not know whether for $X$ with
$|X|\leq 2^{\cont}$ we have the equality  $\langle\Gamma\rangle={\rm{CHom}} (G,X)$ as well.
}

\end{remark}
\begin{remark}\label{OQ2}{\em It follows from Proposition \ref{descr1}$(a_2)$ that on $\S^N$ the box topology $\mathfrak b$ is strictly finer than the uniform topology $\mathfrak u$. This implies that we have the set-theoretic inclusion
\begin{equation}\label{ub}
\left(\S^N,  \mathfrak u \right)^\wedge \subset \left(\S^N,  \mathfrak b \right)^\wedge\,.
\end{equation}
From Fact \ref{kak} we get:
\begin{equation}\label{ub2}
|{\rm{Hom}} (\S^N,\S)|=2^\cont\,.
\end{equation}
From (\ref{ub}), (\ref{ub2}) and Theorem \ref{me-th}$(c)$ we obtain:
\begin{equation}\label{ub3}
|\left(\S^N,  \mathfrak u \right)^\wedge|=2^\cont=|\left(\S^N,  \mathfrak b \right)^\wedge|\,.
\end{equation}
This equality shows that from the pure cardinality arguments it is not possible to conclude  that in (\ref{ub}) we have the strict inclusion.
Nevertheless, we conjecture that the inclusion in (\ref{ub}) is strict, i.e., $\left(\S^N,  \mathfrak u \right)^\wedge \ne  \left(\S^N,  \mathfrak b \right)^\wedge$.

}
\end{remark}

{\bf Acknowledgements.} We are grateful to S. S. Gabriyelyan (Ben-Gurion University of the Negev,
  Beer-Sheva,  Israel) for useful discussions while preparing    this manuscript.


\par

{\bf Addresses:}\\
\\
Dikran Dikranjan\\
 Dipartimento di Matematica e Informatica,
 Universit\`{a} di Udine, Via delle Scienze  206,
  Localit\` a Rizzi
  33100 Udine, Italy\\
 e-mail:  dikran.dikranjan@uniud.it\\
 \\
Elena Mart\'in-Peinador\\
Departamento de Geometr\'{\i}a y Topolog\'{\i}a, Universidad Complutense de Madrid,\\
28040 Madrid,\\
 Spain\\
e-mail: em\_peinador@mat.ucm.es\\
\\
Vaja Tarieladze\\
 Niko Muskhelishvili Institute of Computational Mathematics\\
 of the Georgian Technical University\\
 8, Akuri str. 0160 Tbilisi,\\
   Georgia\\
  e-mail: vajatarieladze@yahoo.com

\end{document}